\documentclass[11pt,a4paper]{article}
\usepackage{amscd}
\usepackage{enumerate}

\usepackage{amsmath, amssymb, latexsym}
\usepackage{bbm}
\usepackage{amsfonts}
\usepackage{amsthm}
\usepackage{relsize}
\usepackage{setspace}
\usepackage{geometry}
\usepackage{url}
\usepackage{xspace}
\usepackage{tocloft}
\usepackage{graphics}
\usepackage{graphicx}
\usepackage{lscape}
\usepackage{microtype}
\usepackage{ulem}

\usepackage[usenames, dvipsnames]{color}
\usepackage[utf8]{inputenc}
\usepackage{tikz}

\usepackage[pagebackref=true]{hyperref}
\usepackage[alphabetic]{amsrefs}
\usepackage[english]{babel}

\usepackage{authblk}

\newtheorem{theorem}{Theorem}[section]
\newtheorem{proposition}[theorem]{Proposition}

\newtheorem{lemma}[theorem]{Lemma}

\theoremstyle{definition}
\newtheorem{definition}[theorem]{Definition}

\theoremstyle{problem}

\newcommand{\Aut}{\mathrm{Aut}}

\newcommand{\QQ}{\mathbf{Q}}
\newcommand{\CC}{\mathbf{C}}
\newcommand{\NN}{\mathbf{N}}

\newcommand{\bd}{\partial}

\newcommand{\SL}{\operatorname{SL}}

\newcommand{\eqcomma}{\ensuremath{\, \text{,}}} 
\newcommand{\conv}{\ensuremath{\mathrm{conv}}}
\newcommand{\rmd}{\ensuremath{\mathrm{d}}}
\newcommand{\cspan}{\ensuremath{\mathop{\overline{\mathrm{span}}}}}
\newcommand{\contc}{\ensuremath{\mathrm{C}_\mathrm{c}}}
\newcommand{\dimvec}{\ensuremath{\mathop{\mathrm{dim}}}}

\def\og{\leavevmode\raise.3ex\hbox{$\scriptscriptstyle\langle\!\langle$~}}
\def\fg{\leavevmode\raise.3ex\hbox{~$\!\scriptscriptstyle\,\rangle\!\rangle$}}

\title{A note on type~${\rm I}$ groups acting on $d$--regular trees}

\author[1]{Corina Ciobotaru\thanks{corina.ciobotaru@unige.ch}}

\date{\today}

\begin{document}
\newcounter{qcounter}

\maketitle

\begin{abstract}
In this note, we give new examples of type~${\rm I}$ groups generalizing a previous result of Ol'shanskii~\cite{olshanskii80}.  More precisely, we prove that all closed non-compact subgroups of $\Aut(T_{d})$ acting transitively on the vertices and on the boundary of a $d$-regular tree and satisfying Tits' independence property are type~${\rm I}$ groups.  We claim no originality as we use standard ingredients: the polar decomposition of those groups and the admissibility of all their irreducible unitary representations.
\end{abstract}

\section{Preliminaries: groups acting on trees}
\label{sec:groups-acting-on-trees}

Denote by $T$ an infinite locally finite tree and by $T_{d}$ the $d$-regular tree with $d \geq 3$. Moreover we say that a tree is thick if the degree of every vertex is at least $3$.  Remark that $T$ is endowed with a canonical metric structure, where an edge is considered of length one, and so with a topology.  The geodesic between two vertices of $T$ is the shortest path in $T$ connecting them.
The automorphism group $\Aut(T)$ is a totally disconnected locally compact group equipped with the topology of pointwise convergence.  If $G \leq \Aut(T)$ is a closed subgroup and $S \subset T$ is some set of vertices, we define the pointwise stabilizer subgroup $G_{S}:=\{ g \in G\; | \; g(s)=s  \text{ for all } s \in S\}$.

In what follows, we denote by $\bd T$ the boundary of the tree $T$, meaning that after fixing a vertex $x_{0}$ in $T$, a boundary point in $\bd T$ is associated with a unique infinite geodesic ray starting at $x_{0}$.  The topology on $T$ can be extended to $\bd T$ as follows.  Let $x_{0}$ be a fixed vertex of $T$.  Let $\xi \in \bd T$ and corresponding to $\xi$  let $\{x_{n}\}_{n \geq 0}$ be the unique infinite sequence of vertices in $T$ starting at $x_{0}$ and determining the infinite geodesic.   A basic neighborhood of $\xi$ is the set $U(x_{0},\xi, n):=\{ x \in T \cup \bd T \; | \; [x_{0}, x_{n}] \subset [x_{0},x)\}$, where by $[y,z]$ we denote the geodesic segment in $T$ between two vertices $y,z$ in $T$.  The topology on $T \cup \bd T$ with basis given by the open balls $B(x,r) \subset T$ together with the collections of all sets $\{U(x_{0}, \xi, n)\}_{\xi \in \bd T, n>0}$ is called the cone topology on $T \cup \bd T$.  This topology does not depend of the choice of the vertex $x_{0}$.  If $T$ is a locally finite tree, then $T \cup \bd T$ is compact in this cone topology.

For what follows we fix a vertex $x_0$ of $T$ and we denote by $B_{n}(x_0)$ be the closed ball of radius $n$ and of center $x_0$.  To simplify the notation for what follows, we denote $U_{n}:=G_{B_{n}(x_0)}$, where $G$ is a closed subgroup of $\Aut(T)$.  Remark that $U_{n}$ is a compact open subgroup of $G$, which is normal in $G_{x_0}$.

\medskip
For the convenience of the reader, we recall below some basic facts from the theory of groups acting by automorphisms on locally finite trees.  We start with the following well known equivalences.

\begin{proposition}[See \cite{FigaNebbia}, Lemma~3.1.1 of~\cite{BM00a} and Corollary~3.6~\cite{CaCi}]
\label{prop:characterisation-boundary-two-transitive-subgroup}
Let $T$ be a locally finite thick tree and let $G \leq \Aut(T)$ be a closed subgroup. Then the following are equivalent:
\begin{enumerate}
\item
  $G_{x}$ is transitive on $\bd T$, for every $x \in T$;
\item
  $G$ is non-compact and there exists $x \in T$ such that $G_{x}$ acts transitively on $\bd T$;
\item
  $G$ is non-compact and acts transitively on $\bd T$;
\item
  $G$ is $2$-transitive on $\bd T$.
\end{enumerate}
\end{proposition}

We also have the following proposition from~\cite[Proposition~10.2]{FigaNebbia}.
\begin{proposition}[See Proposition 10.2 of \cite{FigaNebbia}]
\label{prop:characterisation-boundary-transitive-subgroup}
Let $T$ be a locally finite thick tree and $G$ be a closed non-compact subgroup of $\Aut(T)$ that acts transitively on the boundary. Then $G$ is transitive on the vertices or it is transitive on the edges of $T$. In particular $T$ is a regular or a bi-regular tree.
\end{proposition}

Another well known fact, that one can easily deduce, is the polar decomposition of a closed non-compact subgroup $G \leq \Aut(T_{d})$ acting transitively on the vertices and on the boundary of the tree: $G=G_{x} \langle a \rangle G_{x}$, where $a \in G$ is a hyperbolic automorphism of translation length one and $x$ is a vertex contained in the translation axis of $a$ (see~\cite{BM00b}). 

\section{Type I groups}
Let us introduce the main definition of this note.


\begin{definition}[Definition 5.4.2 in \cite{dixmier77}]
\label{def:type-I-group}
A locally compact group $G$ is of \textbf{type~${\rm I}$}, if for every unitary representation $\pi$ of $G$, the von Neumann algebra $\pi(G)''$ generated by $\pi$ is of type~${\rm I}$.  
\end{definition}

Previously known examples of type~${\rm I}$ groups are all reductive p-adic Lie groups (see \cite{bernstein74-type-I}) and $\Aut(T_{d})$, the automorphisms group of a $d$-regular tree (see Ol'shanski~\cite{olshanskii80} and Demir~\cite{demir04}). 


A useful criterion for a locally compact group to be of type~${\rm I}$ is the following:

\begin{theorem}[See Definition 4.2.1, Proposition 4.3.4 and Theorems 5.5.2  and 15.5.2 in \cite{dixmier77}]
\label{thm:criterion-type-I-group}
Let $G$ be a locally compact group and $K$ a compact subgroup of $G$. Suppose that for any irreducible unitary representation $\pi$ of $G$ and any irreducible unitary representation $\sigma$ of $K$, $\sigma$ occurs with finite multiplicity in $\pi|_K$. Then $G$ is of type~${\rm I}$.
\end{theorem}

Before proceeding with the theory of unitary representations of groups acting on trees, let us give the following easy general lemma:

\begin{lemma}
\label{lem:fixed-vectors}
Let $(\pi,\mathcal{H})$ be a unitary representation of a totally disconnected locally compact group $G$. Then there is a compact open subgroup $K$ in $G$ such that $\pi|_{K}$ has a fixed vector.
\end{lemma}
\begin{proof}
Let $(K_{n})_{n>0}$ be a neighborhood base of $e \in G$ consisting of compact open subgroups. Take $\xi \in \mathcal{H}$ with norm one. Then there is $n \in \mathbb{N}$ such that $K_{n} \cdot \xi \subset B(\xi,1/2)$. So the convex closure $C:= \overline{\conv}(K_{n} \cdot \xi)$ does not contain $0$. Indeed, for all $c_{i}>0$ such that $\sum\limits_{i}c_{i}=1$ and for all $\eta_{i} \in K_{n} \cdot \xi$, we have 
\[
||\sum\limits_{i}c_{i}\eta_{i}-\xi|| \leq \sum\limits_{i}c_{i}||\eta_{i}-\xi|| \leq 1/2.
\]
Since $K_{n}$ is compact there is a $K_{n}$-fixed point in $C$. By  putting $K := K_n$, this finishes the proof.
\end{proof}

\section{ On unitary representations of groups acting on trees}
For the subgroups of $\Aut(T_{d})$, the theory of unitary representations is very well developed only under some assumptions.  Hence, from now on we assume $G$ to be a closed non-compact subgroup of $\Aut(T_{d})$ acting transitively on vertices and on the boundary of $T_d$.  For a unitary representation $(\pi, \mathcal{H})$ of $G$ we denote by $M_{\pi}$ the set of minimal complete finite subtrees $S$ of $T_{d}$ such that $(\pi, \mathcal{H})$ has a $G_{S}$-invariant non-zero vector.  Using Lemma \ref{lem:fixed-vectors}, one obtains the following, well-known, subdivision of irreducible unitary representations $(\pi, \mathcal{H})$ of $G$:
\begin{enumerate}
\item
  $(\pi, \mathcal{H})$ is super-cuspidal if there is an element of $M_{\pi}$ that is neither a point nor an edge;
\item
  $(\pi, \mathcal{H})$ is special if there exists an element of $M_{\pi}$ that is an edge;
\item
  $(\pi, \mathcal{H})$ is spherical if there exists an element of $M_{\pi}$ that is exactly one vertex.
\end{enumerate}

The cases (ii) and (iii) are described in~\cite[Chapter III, resp. Chapter II]{FigaNebbia}.  There exist only two inequivalent special unitary representations, which are in fact the restriction to $G$ of the special unitary representations of $\Aut(T_{d})$.  Also the spherical unitary representations are nothing but the restriction to $G$ of the spherical unitary representations of $\Aut(T_{d})$. Those unitary representations are completely classified in \cite{FigaNebbia}.  Moreover, from~\cite[Lemma~2.1 of Chapter III, resp., Chapter II]{FigaNebbia} or~\cite[Appendix]{demir04}, we obtain that all special and spherical representations of $G$ are admissible.

\begin{definition}
Let $G \leq \Aut(T_{d})$ and $(\pi, \mathcal{H})$ be a unitary representation of $G$. We say that $(\pi, \mathcal{H})$ is \textbf{admissible} if for every compact open subgroup $K$ of $G$ the dimension of the subspace of all $K$-invariant non-zero vectors of $\mathcal{H}$ is finite.
\end{definition}

To study the remaining case of super-cuspidal representations we add an extra condition on the group $G$.  It is called Tits' independence property.  This property guarantees the existence of `enough' rotations in the group $G$.  It was used in the work of Tits~\cite{Ti70}) as a sufficient condition to prove simplicity for `large' subgroups of $\Aut(T_{d})$.  Also in his thesis, Amann~\cite[Theorem~2]{Amann} used it to give a complete classification of all super-cuspidal representations of a closed subgroup in $\Aut(T_{d})$ acting transitively on the vertices and on the boundary of $T_d$ and having Tits' independence property. For a closed subgroup of $\Aut(T_{d})$ acting transitively on the vertices and on the boundary of $T_d$ but without Tits' independence property less is known about the complete classification of all super-cuspidal representations or even about their admissibility.

\begin{definition}
Let $G \leq \Aut(T_d)$.  We say that $G$ has \textbf{Tits' independence property} if for every edge $e$ of $T$ we have the equality $G_{e}=G_{h_{1}} G_{h_{2}}$, where $h_{i}$ are the two half sub-trees emanating from the edge $e$.
\end{definition}

Examples of groups having this property are of course the groups $\Aut(T_{d})$ as well as the universal groups $U(F)$ introduced by Burger--Mozes in~\cite{BM00a, BM00b}. A group without this property is $\SL(2,\QQ_{p})$, for example.

\medskip
As mentioned above, super-cuspidal unitary representations of closed subgroup $G \leq \Aut(T_{d})$, with Tits' independence property,  are studied in Amann~\cite{Amann}. From there \cite[Chapter 2]{Amann} we know that if $(\pi, \mathcal{H})$ is a super-cuspidal representation of $G$ with a non-zero $G_{S}$-invariant vector $v$ and $S \in M_{\pi}$, then the matrix coefficient $g \mapsto \langle \pi(g)v, u\rangle$ has compact support, for every non-zero $G_{S'}$-invariant vector $u$, where $S'$ is a complete finite subtree of $T_{d}$. Moreover, by Demir~\cite[Theorem~3.4.]{demir04} we have:

\begin{theorem}{\cite[Theorem. 3.4.]{demir04}}
\label{thm:adm:G}
Let $G$ be a closed non-compact subgroup of $\Aut(T_{d})$ acting transitively on the vertices and on the boundary of $T_d$ and having Tits' independence property. Then any super-cuspidal representation of $G$ is admissible.
\end{theorem}

For the convenience of the reader, we include the proof from \cite{demir04}.

\begin{proof}
Let $(\pi, \mathcal{H})$ be a super-cuspidal representation of $G$ and recall that $\{U_n\}_{n}$ is a neighborhood basis of $\{e\}$ consisting of compact open subgroups.  It suffices to show that $\dim (\mathcal{H}^{U_n}) < + \infty$ for all $n \in \NN$.

So fix $n \in \NN$ and take $v \in \mathcal{H} \setminus \{0\}$.  We define the $U_n$-invariant function $f: G \to \mathcal{H}^{U_{n}}$ by $f(g):=\int_{U_{n}}\pi(kg)(v) \rmd \mu(k)$, where $\mu$ is the Haar measure on $U_{n}$.  Assume that $\dim(\mathcal{H}^{U_{n}})$ is infinite.  Since $\pi $ is irreducible, we have $\cspan{\pi(G)v}=\mathcal{H}$.  For $g \in G$, the element $f(g)$ is the orthogonal projection of $\pi(g)v$ onto $\mathcal{H}^{U_n}$.  These facts together imply that there is a sequence $\{g_k\}_{k \in \mathbb{N}} \subset G$ such that $\{f(g_k)\}_{k \in \mathbb{N}} \subset \mathcal{H}^{U_{n}}$ is linearly independent.  As $f$ is $U_n$-invariant, the sequence $\{g_k\}_{k \in \mathbb{N}}$ is discrete and hence it goes to infinity in $G$.  

Let $E$ be the closed subspace of $\mathcal{H}^{U_{n}}$ generated by $\{f(g_k)\}_{k \in \mathbb{N}}$.  By Gram-Schmidt orthogonalization one obtains a sequence of orthonormal vectors $\{w_k\}_{k \in \mathbb{N}}$ such that for every $K \in \NN$ the vector spaces spanned by $\{f(g_k)\}_{k < K}$ and by $\{w_k\}_{k < K}$ are the same.  Define the vector $w :=  \sum_k \frac{1}{k} w_k$.  We claim that this vector $w$ has non-zero scalar product with infinitely many vectors from $\{f(g_k)\}_{k \in \NN}$.  Indeed, if this is not the case, then $w$ would be in the span of finitely many vectors $\{f(g_k)\}_{k \in \NN}$.  But then $w$ would be in the span of finitely many vectors $\{w_k\}_{k \in \NN}$ too.  This is a contradiction.

By \cite[Chapter 2]{Amann}, the function $g \mapsto \langle \pi(g) v, w \rangle$ is compactly supported.  Since
\begin{equation*}
  \langle  f(g), w \rangle = \int_{U_n} \langle \pi(k g)(v) , w \rangle \rmd \mu (k)
  \eqcomma
\end{equation*}
we have that $f$ is also compactly supported.  This contradicts $\langle f(g_k) , w \rangle \neq 0$ for infinitely many $k \in \NN$.  We have finished the proof.
\end{proof}

Summarizing the previous discussion, we showed:
\begin{theorem}
\label{thm:adm:com:G}
 Let $G \leq \Aut(T_{d})$ be a closed non-compact subgroup acting transitively on the vertices and on the boundary of $T_d$ and having Tits' independence property. Then all its irreducible unitary representations are admissible.
\end{theorem}

To the best of our knowledge those are the only known examples of totally disconnected locally compact groups for which all irreducible unitary representations are admissible.

We are now ready to prove the main theorem of this note.

\begin{theorem}
\label{thm:type-I-groups}
 Any closed non-compact subgroup of $\Aut(T_{d})$ acting transitively on the vertices and on the boundary of the tree and having Tits' independence property is of type~${\rm I}$.
\end{theorem}

\begin{proof}
To prove the theorem we verify the criterion from Theorem~\ref{thm:criterion-type-I-group} for our group $G$ and $K=G_{x}$, with $x$ a vertex in $T_{d}$. 

So let $(\pi,\mathcal{H})$ be an irreducible unitary representation of $G$. As $K$ is compact, we have the direct sum decomposition $\pi|_{K}=\bigoplus_{\pi_{i} } n_{i} \pi_{i}$, where $(\pi_{i})_{i}$ are pairwise non-isomorphic finite dimensional irreducible unitary representations of $K$ and $n_{i}$ is the multiplicity of $\pi_{i}$ in $\pi|_{K}$.  Suppose that $n_{i}=\infty$ for some $\pi_{i}$.  By Lemma~\ref{lem:fixed-vectors} applied to the compact subgroup $K$ there exists a compact open subgroup $K_{n} \subset K$ such that $\pi_{i}|_{K_{n}}$ has a fixed vector.  Therefore there is an infinite number of $K_{n}$-fixed vectors for the representation $\pi$ of $G$.  Invoking Theorem~\ref{thm:adm:com:G}, we obtain a contradiction.  This verifies the criterion of Theorem~\ref{thm:criterion-type-I-group} and finishes the proof.
\end{proof}

Let us finally give a slight strengthening of Theorem \ref{thm:adm:com:G}, which corresponds to the formulation of the main theorem of Bernstein~\cite{bernstein74-type-I} and of Demir~\cite{demir04}.

\begin{theorem}{\cite{bernstein74-type-I} and Theorem 3.3 in \cite{demir04}}
\label{thm:bounded:dim}
Let $G$ a closed non-compact subgroups of $\Aut(T_{d})$ acting transitively on the vertices and on the boundary of $T_d$ and having Tits' independence property.

For every compact open subgroup $K \leq G$ there exists $N=N(G,K)$ such that for every $(\pi, \mathcal{H})$ an irreducible unitary representation of $G$ we have $\dim(\mathcal{H}^K) \leq  N$.
\end{theorem}

\begin{proof}
It suffices to check the statement of the theorem for $K = U_n$, $n \in \NN$.  So fix $n \in \NN$.

Recall from \cite[Lemma~4.1]{FigaNebbia} that the  Hecke algebra $\contc(G,G_{x})$ is commutative under the convolution product.  As $U_{n}$ is a normal finite index subgroup of $G_{x}$, we have that $\dimvec(\contc(G_{x},U_{n})) = \dimvec (\CC[G_x/U_n]) = [G_{x}:~U_{n}]$.  Following the same calculations as in \cite{bernstein74-type-I} or \cite{demir04} by using the polar decomposition $G = G_{x}\langle a \rangle G_{x}$, one obtains Theorem 3.2 from \cite{demir04}: For a closed non-compact subgroup $G \leq \Aut(T_{d})$ acting transitively on the vertices and on the boundary of $T_d$, we have $\contc(G,U_{n})=\contc(G_{x},U_{n}) A \; \contc(G_{x},U_{n})$, where $A$ is the subalgebra generated by the characteristic function of $U_{n} a U_{n}$. 

By the general representation theory of Hecke algebras we know the following. Every irreducible unitary representation $(\pi, \mathcal{H})$ of our group $G$ induces an irreducible $*$-representation of the Hecke algebra $\contc(G,K)$ on the space of $K$-invariant vectors $\mathcal{H}^{K}$, where $K$ denotes a compact open subgroup of $G$. Thus, using the admissibility of all irreducible unitary representations of our group $G$, to obtain our conclusion it is enough to check that the dimension of every finite dimensional irreducible representation of the Hecke algebra $\contc(G,U_{n})$ is bounded above by a constant. Now the proof of \cite[Theorem 3.3]{demir04} (see also \cite{olshanskii80}) gives the constant $N$.  This finishes the proof.

\end{proof}

\end{document}